\documentclass{article}

\usepackage{arxiv}

\usepackage[utf8]{inputenc} 
\usepackage[T1]{fontenc}    
\usepackage{hyperref}       
\usepackage{url}            
\usepackage{booktabs}       
\usepackage{amsfonts}       
\usepackage{nicefrac}       
\usepackage{microtype}      

\usepackage{amsmath}
\usepackage{amsthm}
\usepackage{graphicx}

\usepackage{multirow}
\usepackage{array}

\usepackage{amsthm}

\usepackage{amsmath}
\usepackage{mathtools}
\DeclarePairedDelimiter{\ceil}{\lceil}{\rceil}

\newtheorem{theorem}{Theorem}[section]

\newtheorem{lemma}[theorem]{Lemma}

\theoremstyle{definition}

\theoremstyle{remark}

\usepackage{natbib}
\setcitestyle{square,numbers}

\title{Uniqueness of a Median of a Binomial Distribution with Rational Probability}

\author{
  Szymon Nowakowski
    \\
  AI Investments\\
  \texttt{szymon.nowakowski@aiinvestments.pl} \\
}

\begin{document}
\maketitle

\begin{abstract}
In this paper we show that the median of the binomial distribution $B(n,p)$ is unique for all rational $p$, with the only exception of $p=\frac{1}{2}$ and $n$ odd.
\end{abstract}


\section{Introduction}
The tail estimates of probabilistic distributions have been the very active research frontier in theoretic (large deviation theory) and applied sciences, both. Not so much work has been devoted to medians, and the existing results seem to be either not that well known or misunderstood. As the example, the famous mean, median and mode inequality theorems \cite{Runnenburg1978-zw,Van_Zwet1979-rv,Dharmadhikari1983-el,Abdous1998-st} are a relatively new achievement. The discrete case \cite{Abdous1998-st} is the work as recent as the turn of the centuries. Those theorems establish the order of the mean, medians and modes under certain conditions. Nevertheless, there seems to be a widespread belief that the order holds universally, and there exist papers devoted to correcting that belief \cite{Abadir2005-bt,Von_Hippel2005-of}. 

Recently, there have been papers that establish the bounds on a distance of a median to the mean in the case of a binomial distribution (\cite{Kaas1980-fn}), a binomial and negative binomial distribution (\cite{Gob1994-sv}) and in the case of a binomial and Poisson distributions (\cite{Hamza1995-mb}). The first of those papers introduces the notion of a \emph{weak} median (a non-unique median) and the \emph{strong} (unique) median. In the same paper sufficient conditions are given for the median to be unique in the binomial distribution. The sufficient conditions are based on a distance between the median candidate and the mean of the binomial distribution.

In this paper we state another sufficient condition of a completely different nature than before (not based on a distance between median candidates and the mean): we show that the median of the binomial distribution $B(n,p)$ is unique (strong in the notation of \cite{Kaas1980-fn}) for all rational $p$, with the only exception of $p=\frac{1}{2}$ and $n$ odd. 

The main motivation of this paper comes from considering an approximation of a hypergeometric distribution with a binomial distribution, with $p$ being a ratio between certain integer parameters of the hypergeometric distribution being approximated, thus $p$ being rational. 
In toy problems such as urn problems (drawing with replacement), $p$ is a ratio between integers (numbers of balls). In practical statistical inference, $p$ is frequently a ratio between integer counts, too. In those cases one may apply the results of this paper: that there is the unique median guaranteed to exist (with the only exception of $p=\frac{1}{2}$ and $n$ odd). 

\section{Notation}
\label{sec:notation}
Binomial distribution $B(n,p)$ with parameters $n \in \mathbb{N}$ and $0 \leq p \leq 1$ is defined with the use of its discrete probability density function $b(\cdot,n,p)$
\begin{equation}
b(k,n,p) = \binom{n}{k}p^k (1-p)^{n-k}\text{, for } k \in \mathbb{N},\,0 \leq k \leq n
\end{equation}
Its cumulative distribution function will be denoted $B(\cdot,n,p)$, i.e.
\begin{equation}
    B(k,n,p) = \sum_{i=0}^k b(i,n,p)\text{, for } k \in \mathbb{N},\,0 \leq k \leq n
\end{equation}
We call a discrete random variable $X$ binomial with a distribution $B(n,p)$ iif 
\begin{equation}
    P(X \leq k) = B(k,n,p)\text{, for all } k \in \mathbb{N},\,0 \leq k \leq n
\end{equation}
We say that $m \in \mathbb{R}$ is a median of a discretely distributed random variable $X$ iif 
\begin{equation}\label{eq:median_def}
    P(X \leq m) \geq \frac{1}{2}\text{ and }P(X \geq m) \geq \frac{1}{2}
\end{equation}
Following \cite{Kaas1980-fn} we make an additional distinction here: we call $m\in \mathbb{R}$ the unique median iif 
\begin{equation}\label{eq:strong_median_def}
    P(X \leq m) > \frac{1}{2}\text{ and }P(X \geq m) > \frac{1}{2}
\end{equation}
Please note, that we do not need to restrict the last two definitions to binomial distributions only.

\section{Prerequisites}
Before we move to examine the main result of this paper, let us state a few facts to lay the foundation of our understanding of the medians in the discrete distributions.
\begin{lemma}\label{lemma:unique_median_in_support}
If $m \in \mathbb{R}$ is the unique median of a discrete random variable $X$, then $m$ belongs to the support of $X$.
\end{lemma}
\begin{proof}
From the definition of the strong median (\ref{eq:strong_median_def}) we have
\begin{equation*}
    P(X > m) = 1- P(X \leq m) < \frac{1}{2}\text{ and }P(X \geq m) > \frac{1}{2}
\end{equation*}
and thus $P(X = m) >0$.

\end{proof}

\begin{lemma}\label{lemma:median_iif_1_2_condition}
If $m \in \mathbb{R}$, $X$ is a discrete random variable then
\begin{equation}\label{eq:one_half}
    P(X \leq m)= \frac{1}{2}\text{ or }P(X \geq m) = \frac{1}{2}
\end{equation}
if and only if $m$ is a median but not the unique median of $X$.
\end{lemma}
\begin{proof}
$(\implies)$:
Suppose that the first equation in (\ref{eq:one_half}) holds: we have
\begin{equation}
    P(X \geq m) \geq P(X > m)= 1-P(X \leq m) = \frac{1}{2}
\end{equation}
and thus $m$ is a median by (\ref{eq:median_def}).
By the symmetrical argument for the second equation in (\ref{eq:one_half}) we can show that $m$ satisfies the definition of a median (\ref{eq:median_def}) in this case too. Obviously, $m$ fails (\ref{eq:strong_median_def}) and thus it is not the unique median.

$(\impliedby)$:
Obviously, if (\ref{eq:median_def}) holds but any one of the two inequalities (\ref{eq:strong_median_def}) fails, then (\ref{eq:one_half}) must hold. 

This concludes the proof.
\end{proof}

\begin{lemma}\label{lemma:median_not_unique}
If $m\in \mathbb{R}$ is a median but not the unique median of a discrete random variable $X$ with its support being a closed set in $\mathbb{R}$, then there exist uniquely defined $m_1<m_2$, both in the support of $X$, such that $m \in [m_1, m_2]$ and all $m' \in [m_1, m_2]$ are medians. Additionally we have
\begin{equation}
    P(X \leq m_1) = P(X \geq m_2) = \frac{1}{2}
\end{equation}
\end{lemma}
\begin{proof}
To see that please note, that by the Lemma \ref{lemma:median_iif_1_2_condition} we have 
\begin{equation}
    P(X \leq m) = \frac{1}{2}\text{ or }P(X \geq m) = \frac{1}{2}
\end{equation}
Without losing generality assume that the first equality holds, the argument for the second equality is symmetrical. 

If $m$ is in the support of $X$ then let $m_1 = m$. Otherwise, if $m$ is not in the support of $X$, select $m_1<m$ from the support of $X$ for which 
$P(X \leq m_1) = \frac{1}{2}$ still holds. It is possible because the support of $X$ is a closed set in $\mathbb{R}$.
Either way, we have found $m_1 \leq m$ with 
$P(X \leq m_1) = \frac{1}{2}$. Such a point $m_1$ is uniquely defined.

To find $m_2$ observe, that we have
\begin{equation}
    P(X > m) = 1-P(X \leq m) = \frac{1}{2}
\end{equation}
which allows us to find $m_2 > m \geq m_1$, $m_2$ is from the support of $X$ with the property $P(X \geq m_2) = \frac{1}{2}$. Again, we use the property that the support of $X$ is a closed set in $\mathbb{R}$. Such a point $m_2$ is uniquely defined, too.

For all $m' \in [m_1, m_2]$ we have 
\begin{equation}
    P(X \leq m') \geq P(x \leq m_1)= \frac{1}{2}\text{ and }P(X \geq m') \geq P(x \geq m_2)= \frac{1}{2}
\end{equation}
Thus $[m_1, m_2]$ is the uniquely defined interval of medians of $X$.
\end{proof}

\begin{lemma}\label{lemma:unique_median_unique}
If $m_1, m_2 \in \mathbb{R}$ are the unique medians of a discrete random variable $X$, then $m_1=m_2$.
\end{lemma}
\begin{proof}
The proof will proceed by showing contradiction. Let us assume that $m_1 \neq m_2$. Without losing generality we can assume that $m_1<m_2$. The case $m_1>m_2$ is symmetrical.

By the definition of the unique median (\ref{eq:strong_median_def}) for $m_2$ we have
\begin{equation}
P(X \leq m_1) \leq P(X < m_2) = 1-P(X \geq m_2)< \frac{1}{2}
\end{equation}
Thus, $m_1$ doesn't satisfy (\ref{eq:strong_median_def}) and thus it is not the unique median, showing contradiction.
\end{proof}

Lemmas \ref{lemma:median_not_unique} and \ref{lemma:unique_median_unique} explain the motivation behind calling a median satisfying (\ref{eq:strong_median_def}) the \emph{unique} median. Lemmas \ref{lemma:unique_median_in_support} and \ref{lemma:median_not_unique} state additionally that this distinction (unique/non-unique) holds in restriction to a support of a discrete random variable with a closed support, too. 

\section{Main result}

Shifting the discussion back to binomial distributions, we can see that the discriminating condition for $B(n,p)$ to have a non-unique median is having the value $k-1$ in the support with $B(k-1, n, p) = \frac{1}{2}$. Then and only then $k-1$ and $k$ are the two ends of the interval of non-unique medians. Otherwise, if there is no such $k-1$ with $B(k-1, n, p) = \frac{1}{2}$, there is the unique median in $B(n,p)$.

The above observation gives rise to examining, for each $n \geq 1$, the $n$ values $p_{n, 1}, \ldots, p_{n, n}$ defined in \cite{Kaas1980-fn} as follows:
\begin{equation}\label{eq:p_i}
    B(k-1, n, p_{n, k}) = \frac{1}{2}\text{, for } k \in \mathbb{N},\,1 \leq k \leq n
\end{equation}
Such $p_{n,k}$ exist, because for fixed $k$ and $n$, $B(k-1, n, p)$ is a continuous function of $p$ with values $B(k-1, n, p=0)=1$ and $B(k-1, n, p=1)=0$, thus $\frac{1}{2}$ must be attained by Darboux's theorem at some point $p \in (0,1)$, too. 

In other words, for $1 \leq k \leq n$, $B(n,p_{n, k})$ is the distribution with $k-1$ and $k$ as the two ends of the interval of non-unique medians. 

\begin{theorem}
    Fix $n \geq 1$. The $n$ values $p_{n, 1}, \ldots, p_{n, n}$ defined as in (\ref{eq:p_i}) are irrational, with the only exception $p_{n, \ceil{\frac{n}{2}}}=\frac{1}{2}$ for odd $n$.
\end{theorem}
\begin{proof}
The proof will proceed in five parts marked with Roman numerals I-V:

\textbf{I.} $p_{n, i}<p_{n, j}$ for $1 \leq i <j \leq n$:

We can see that $n \geq 2$. $B(k, n, p)$ as a binomial CDF is a strictly increasing function of $k$ with $n$ and $p$ fixed and it is a strictly decreasing function of $p$ with $n$ and $k$ fixed. To see that, fix $n \geq 2$ and $0 \leq j \leq n-1$ and calculate the derivative
\begin{equation}
\begin{split}
    \frac{d}{dp}B(j,n,p)=&\sum_{i=0}^j  \binom{n}{i} \bigg( i p^{i-1}(1-p)^{n-i}-p^i(n-i) (1-p)^{n-i-1} \bigg)=\\
    =&\sum_{i=0}^{j-1}\binom{n}{i+1}(i+1)p^i(1-p)^{n-i-1}
    - \sum_{i=0}^{j}\binom{n}{i}(n-i)p^i(1-p)^{n-i-1}
\end{split}
\end{equation}
It is easy to check that $\binom{n}{i+1}(i+1) = \binom{n}{i}(n-i) = \binom{n-1}{i}n$. It allows telescopic cancelling of respective terms in both sums and we are left only with the last term of the right sum, i.e.
\begin{equation}
    \frac{d}{dp}B(j,n,p) = -n\binom{n-1}{j}p^j(1-p)^{(n-1)-j} = -n \cdot b(j,n-1,p) < 0
\end{equation}
For $i<j$ we have
\begin{equation}
    B(j-1,n,p_{n, j}) = \frac{1}{2}=B(i-1,n,p_{n, i})<B(j-1,n,p_{n, i})
\end{equation}
So we have  $p_{n, i}<p_{n, j}$, which ends this part of the proof.

\textbf{II.} $0< p_{n, i} = 1 - p_{n, (n-i+1)}<1$ for $1 \leq i \leq n$:

First note, that if any of $p_{n,i}$ were equal 0 or 1, the distribution would degenerate and CDF in no point would be equal to $\frac{1}{2}$.

To prove the equation $p_{n, i} = 1 - p_{n, (n-i+1)}$, observe that for random variables $X_+$ with a binomial distribution $B(n,p_{n,i})$ and $X_-$ with a binomial distribution $B(n,1-p_{n,i})$ we have
\begin{equation}
    \frac{1}{2} = B(i-1,n,p_{n,i}) = P(X_+ \leq i-1) = P(X_- \geq n-(i-1))
\end{equation}
Thus,
\begin{equation}
\begin{split}
    B(n-i, n, 1-p_{n,i}) =& P(X_- \leq n-i) = P(X_- < n-i +1) =1-P(X_- \geq n-i+1) = \frac{1}{2} = \\
    =&B(n-i, n, p_{n,(n-i+1)}) 
\end{split}
\end{equation}
It ends this part of the proof.
\end{proof}

\textbf{III.} $p_{n, \ceil{\frac{n}{2}}}=\frac{1}{2}$ for odd $n$:

This is a straightforward consequence of part II., after you note that $ \ceil{\frac{n}{2}} = n - \ceil{\frac{n}{2}} +1$ for odd $n$.

\textbf{IV.} $p_{n, i}$ is irrational for $i>\ceil{\frac{n}{2}}$:

First, observe that for odd $n$, by using parts I. and III., we have $p_{n,i}>p_{n,\ceil{\frac{n}{2}}} = \frac{1}{2}$ for $i>\ceil{\frac{n}{2}}$.

For $n$ even, we use parts I. and II. to arrive to the same conclusion: $p_{n, i} > p_{n, (n-i+1)} = 1 - p_{n, i}$, for $i> \frac{n}{2}$, thus $p_{n, i} > \frac{1}{2}$ for $i> \frac{n}{2}= \ceil{\frac{n}{2}}$.

By using part II. again, to get the upper bound on $p_{n,i}$, we arrive to the conclusion, that $p_{n,i} \in (\frac{1}{2},1)$.

We observe that $p=p_{n,i}$ is the solution to the equation $2B(i-1,n,p)-1=0$ which is polynomial in relation to $p$ with all coefficients being integers. We examine the coefficient at $p^0$ after the polynomial $2B(i-1,n,p)-1$ is fully expanded:
\begin{equation}
\big[p^0\big]\Big(2B(i-1,n,p)-1\Big) = \big[p^0\big]\bigg(2 \sum_{j=0}^{i-1} \binom{n}{j}p^j(1-p)^{n-j} \bigg)-1 = \big[p^0\big]\bigg( 2 \binom{n}{0}p^0(1-p)^{n} \bigg)-1=1
\end{equation}
If $p_{n,i}$ were rational of the form $\frac{q}{r}$ with $q$ and $r$ relatively prime integers, then (by the Rational Root Theorem) $q$ would divide $\big[p^0\big]\Big(2B(i-1,n,p)-1\Big)= 1$. So $p_{n,i}$ would be rational of the form $\frac{1}{r}$. But that is impossible, because in $(\frac{1}{2},1)$ interval there are no rational numbers of the form $\frac{1}{r}$ with integer $r$. 

This finalizes this part of the proof.

\textbf{V.} $p_{n, i}$ is irrational for $i \neq \ceil{\frac{n}{2}}$ or $n$ even.

This is a straightforward consequence of parts II. and IV. It also finalizes the proof of the main result of this paper.

\bibliographystyle{unsrtnat}  
\bibliography{references}  

\end{document}